\documentclass[12pt]{amsart}
\usepackage{amsmath,amssymb,amsbsy,amsfonts,amsthm,latexsym,
            amsopn,amstext,amsxtra,euscript,amscd,mathrsfs,array,enumerate}
\usepackage{url}
\usepackage[colorlinks,linkcolor=blue,anchorcolor=blue,citecolor=blue,backref=page]{hyperref}

\usepackage{color}

\usepackage[norefs,nocites]{refcheck}
 
\begin{document}

\newtheorem{theorem}{Theorem}
\newtheorem{lemma}[theorem]{Lemma}
\newtheorem{lem}[theorem]{Lemma}
\newtheorem{thm}[theorem]{Theorem}
\newtheorem{claim}[theorem]{Claim}
\newtheorem{cor}[theorem]{Corollary}
\newtheorem{prop}[theorem]{Proposition}
\newtheorem{definition}{Definition}
\newtheorem{question}[theorem]{Question}
\newtheorem{conj}[theorem]{Conjecture}
\newcommand{\hh}{{{\mathrm h}}}

\numberwithin{equation}{section}
\numberwithin{theorem}{section}

\def\sssum{\mathop{\sum\!\sum\!\sum}}
\def\ssum{\mathop{\sum\ldots \sum}}

\def \balpha{\boldsymbol\alpha}
\def \bbeta{\boldsymbol\beta}
\def \bgamma{{\boldsymbol\gamma}}
\def \bomega{\boldsymbol\omega}

\newcommand{\Res}{\mathrm{Res}\,}
\newcommand{\Gal}{\mathrm{Gal}\,}

\def\sssum{\mathop{\sum\!\sum\!\sum}}
\def\ssum{\mathop{\sum\ldots \sum}}
\def\dsum{\mathop{\sum\  \sum}}
\def\iint{\mathop{\int\ldots \int}}

\def\squareforqed{\hbox{\rlap{$\sqcap$}$\sqcup$}}
\def\qed{\ifmmode\squareforqed\else{\unskip\nobreak\hfil
\penalty50\hskip1em\null\nobreak\hfil\squareforqed
\parfillskip=0pt\finalhyphendemerits=0\endgraf}\fi}%%

%  use the AMS-Euler Fraktur fonts
%%%%%%%%%%%%%%%%%%%%%%%%%%%%%%%%%%
\newfont{\teneufm}{eufm10}
\newfont{\seveneufm}{eufm7}
\newfont{\fiveeufm}{eufm5}
%%%%%%%%%%%%%%%%%%%%%%%%%%%%%%%%%
%
%  allow automatic size selection in math mode
%
%%%%%%%%%%%%%%%%%%%%%%%%%%%%%%%%%
\newfam\eufmfam
     \textfont\eufmfam=\teneufm
\scriptfont\eufmfam=\seveneufm
     \scriptscriptfont\eufmfam=\fiveeufm
%%%%%%%%%%%%%%%%%%%%%%%%%%%%%%%%%
%
%  \frak works on a single symbol at a time...
%
\def\frak#1{{\fam\eufmfam\relax#1}}

\def\fK{\mathfrak K}
\def\fT{\mathfrak{T}}

\def\fA{{\mathfrak A}}
\def\fB{{\mathfrak B}}
\def\fC{{\mathfrak C}}
\def\fD{{\mathfrak D}}
\def\fM{{\mathfrak M}}

\newcommand{\sX}{\ensuremath{\mathscr{X}}}

\def\eqref#1{(\ref{#1})}

\def\vec#1{\mathbf{#1}}
\def\dist{\mathrm{dist}}
\def\vol#1{\mathrm{vol}\,{#1}}

\def\squareforqed{\hbox{\rlap{$\sqcap$}$\sqcup$}}
\def\qed{\ifmmode\squareforqed\else{\unskip\nobreak\hfil
\penalty50\hskip1em\null\nobreak\hfil\squareforqed
\parfillskip=0pt\finalhyphendemerits=0\endgraf}\fi}

\def\sA{\mathscr A}
\def\sB{\mathscr B}
\def\sC{\mathscr C}
\def\sD{\Delta}
\def\sE{\mathscr E}
\def\sF{\mathscr F}
\def\sG{\mathscr G}
\def\sH{\mathscr H}
\def\sI{\mathscr I}
\def\sJ{\mathscr J}
\def\sK{\mathscr K}
\def\sL{\mathscr L}
\def\sM{\mathscr M}
\def\sN{\mathscr N}
\def\sO{\mathscr O}
\def\sP{\mathscr P}
\def\sQ{\mathscr Q}
\def\sR{\mathscr R}
\def\sS{\mathscr S}
\def\sU{\mathscr U}
\def\sT{\mathscr T}
\def\sV{\mathscr V}
\def\sW{\mathscr W}
\def\sX{\mathscr X}
\def\sY{\mathscr Y}
\def\sZ{\mathscr Z}

%%%%%%%%%%%%%%%%%%%%%%%%%
% Alphabet calligraphie %
%%%%%%%%%%%%%%%%%%%%%%%%%
\def\cA{{\mathcal A}}
\def\cB{{\mathcal B}}
\def\cC{{\mathcal C}}
\def\cD{{\mathcal D}}
\def\cE{{\mathcal E}}
\def\cF{{\mathcal F}}
\def\cG{{\mathcal G}}
\def\cH{{\mathcal H}}
\def\cI{{\mathcal I}}
\def\cJ{{\mathcal J}}
\def\cK{{\mathcal K}}
\def\cL{{\mathcal L}}
\def\cM{{\mathcal M}}
\def\cN{{\mathcal N}}
\def\cO{{\mathcal O}}
\def\cP{{\mathcal P}}
\def\cQ{{\mathcal Q}}
\def\cR{{\mathcal R}}
\def\cS{{\mathcal S}}
\def\cT{{\mathcal T}}
\def\cU{{\mathcal U}}
\def\cV{{\mathcal V}}
\def\cW{{\mathcal W}}
\def\cX{{\mathcal X}}
\def\cY{{\mathcal Y}}
\def\cZ{{\mathcal Z}}
\newcommand{\rmod}[1]{\: \mbox{mod} \: #1}

\def\rE{\mathrm E}
\def\rS{\mathrm S}

\def\vr{\mathbf r}

\def\e{{\mathbf{\,e}}}
\def\ep{{\mathbf{\,e}}_p}
\def\em{{\mathbf{\,e}}_m}
\def\en{{\mathbf{\,e}}_n}

\def\Tr{{\mathrm{Tr}}}
\def\Nm{{\mathrm{Nm}}}
\def\supp{{\mathrm{supp}}}

\def\gmax{g}
\def\ve{\varepsilon}
\def\rot{\operatorname{rot}}
\def\ord{\operatorname{ord}}

\def\lcm{{\mathrm{lcm}}}
\def\lcm{{\mathrm{lcm}}}

\def \ovFp{\overline{\F}_p}

\def\({\left(}
\def\){\right)}
\def\fl#1{\left\lfloor#1\right\rfloor}
\def\rf#1{\left\lceil#1\right\rceil}

\definecolor{olive}{rgb}{0.3, 0.4, .1}
\definecolor{dgreen}{rgb}{0.,0.5,0.}

\def\mand{\qquad \mbox{and} \qquad}

\newcommand{\commSK}[1]{\marginpar{%
\begin{color}{red}
\vskip-\baselineskip %raise the marginpar a bit
\raggedright\footnotesize
\itshape\hrule \smallskip SK: #1\par\smallskip\hrule\end{color}}}

\newcommand{\commIS}[1]{\marginpar{%
\begin{color}{blue}
\vskip-\baselineskip %raise the marginpar a bit
\raggedright\footnotesize
\itshape\hrule \smallskip IS: #1\par\smallskip\hrule\end{color}}}

\newcommand{\commIIS}[1]{\marginpar{%
\begin{color}{dgreen} 
\vskip-\baselineskip %raise the marginpar a bit
\raggedright\footnotesize
\itshape\hrule \smallskip IS: #1\par\smallskip\hrule\end{color}}}

\newcommand{\commIV}[1]{\marginpar{%
\begin{color}{magenta}
\vskip-\baselineskip %raise the marginpar a bit
\raggedright\footnotesize
\itshape\hrule \smallskip IV.: #1\par\smallskip\hrule\end{color}}}

%%%%%%%%%%%%%%%%%%%%%%%%%%%%%%%%%%%%%%%%%%%%%%%%%%%%%%%%
%%%%%%%%%%%%%%%%%%%%%%%%%%%%%%%%%%%%%%%%%%%%%%%%%%%%%%%%
%%%%%%%%%%%%%%%%%%%%%%%%%%%%%%%%%%%%%%%%%%%%%%%%%%%%%%%%
%%%%%%%%%%%%%%%%%%%%%%%%%%%%%%%%%%%%%%%%%%%%%%%%%%%%%%%%

%%%%%%%  END OF STANDARD STUFF %%%%%%%%%

%%%%%%%%%%%%%%%%%%%%%%%%%%%%%%%%%%%%%%%%%%%%%%%%%%%%%%%%
%%%%%%%%%%%%%%%%%%%%%%%%%%%%%%%%%%%%%%%%%%%%%%%%%%%%%%%%
%%%%%%%%%%%%%%%%%%%%%%%%%%%%%%%%%%%%%%%%%%%%%%%%%%%%%%%%
%%%%%%%%%%%%%%%%%%%%%%%%%%%%%%%%%%%%%%%%%%%%%%%%%%%%%%%
%%%%%%%%%%%
%%% Spell

\hyphenation{re-pub-lished}
\hyphenation{ne-ce-ssa-ry}

\parskip 4pt plus 2pt minus 2pt

\def\bfdefault{b}
\overfullrule=5pt

\def \F{{\mathbb F}}
\def \K{{\mathbb K}}
\def \L{{\mathbb L}}
\def \N{{\mathbb N}}
\def \Z{{\mathbb Z}}
\def \Q{{\mathbb Q}}
\def \R{{\mathbb R}}
\def \C{{\mathbb C}}
\def\Fp{\F_p}
\def \fp{\Fp^*}

\title[Markoff Triples]{On the Structure of Graphs of Markoff Triples}

\author[S. Konyagin]{Sergei V.~Konyagin}
\address{Steklov Mathematical Institute,
8, Gubkin Street, Moscow, 119991, Russia} 
\email{konyagin@mi.ras.ru}

\author[S.  Makarychev]{Sergey V. Makarychev}
\address{Skolkovo Institute of Science and Technology,
Skolkovo Innovation Center, Building 3,
Moscow, 143026,
Russia}
\email{svmakarychev@yandex.ru}

\author[I.   Shparlinski]{Igor E. Shparlinski}
\address{Department of Pure Mathematics, University of New South Wales\\
2052 NSW, Australia.}
\email{igor.shparlinski@unsw.edu.au}

\author[I.   Vyugin]{Ilya V. Vyugin}
\address{Institute for Information Transmission Problems RAS\\
19, Bolshoy Karetny per., Moscow, 127051, Russia\\
and\\
Department of Mathematics, Higher School of Economics\\
6, Usacheva Street, Moscow, 119048, Russia.}
\email{vyugin@gmail.com}

\begin{abstract}  We sharpen the bounds of J.~Bourgain, A.~Gamburd and P.~Sarnak~(2016)
on the possible number of nodes outside  the ``giant component'' and on the size of individual 
connected components in the suitably defined
functional  graph of Markoff triples modulo $p$. This is a step towards the conjecture that 
there are no such nodes at all.  
\end{abstract}


\maketitle

\section{Introduction}

\subsection{Background and motivation}

We recall that the set $\cM$ of {\it Markoff triples\/} $(x, y, z) \in \N^3$ is the 
set of positive integer solutions to the Diophantine equation 
\begin{equation}
\label{eq:Markoff}
x^2 +  y^2 + z^2 = 3 x yz, \qquad (x, y, z)\in \Z^3.
\end{equation}
One easily verifies that the map 
$$\cR_1: (x, y, z) \mapsto (3 yz - x,  y, z)
$$ 
and similarly defined maps $\cR_2$, $\cR_3$ (which are all involutions),  send one Markoff triple to another.
It is also obvious that
 so do permutations  $ \Pi \in \rS_3$ of the components of $(x, y, z)$. 
 
By a classical result of Markoff~\cite{Mar1,Mar2}  one can get all integer solutions to~\eqref{eq:Markoff}
starting from the solution $(1,1,1)$ and then applying the above transformations. More formally, let $\Gamma$ be 
the group of transformations generated by $\cR_1, \cR_2, \cR_3$ and permutations $\Pi \in \rS_3$.
Then the orbit of $(1,1,1)$ under $\Gamma$ contains $\cM$.   
 Hence, if one defines a {\it functional graph\/} on Markoff triples, where, starting from the ``root''  $(1,1,1)$,  
 the edges
 $(x_1, y_1,z_1)  \to  (x_2, y_2,z_2)$ are governed by $(x_2, y_2,z_2) = \cT(x_1, y_1,z_1)$,  where 
\begin{equation}
\label{eq:Transf}
 \cT=\{\cR_1,\cR_2,\cR_3\}\cup \rS_3, 
 \end{equation}
 then this graph is connected.

 Baragar~\cite[Section~V.3]{Bar} and, more recently, 
 Bourgain, Gamburd and Sarnak~\cite{BGS1,BGS2} conjecture that this property is preserved modulo all 
sufficiently large primes and the set of non-zero solutions $\cM_p$
to~\eqref{eq:Markoff} considered modulo $p$ can be obtained from the set of Markoff triples  
$\cM$ reduced modulo $p$.  This conjecture  means that the functional graph $\cX_p$ associated 
with the transformation~\eqref{eq:Transf} remains connected.

Accordingly, if we define by $\cC_p\subseteq \cM_p$ the set of the triples in the largest connected component
of the above graph $\cX_p$,  then we can state:

\begin{conj}[Baragar~\cite{Bar}; Bourgain, Gamburd and Sarnak~\cite{BGS1,BGS2}] 
\label{conj:BGS} 
For every prime $p$  we have $\cC_p = \cM_p$.
\end{conj}

Bourgain, Gamburd and Sarnak~\cite{BGS1,BGS2} have obtained several major results towards 
Conjecture~\ref{conj:BGS}, see also~\cite{CGMP,dCIL,dCIM,GMR}. 
For example, by~\cite[Theorem~1]{BGS1} we have 
\begin{equation}
\label{eq:except set}
\#\(\cM_p \setminus \cC_p\) = p ^{o(1)}, \qquad \text{as} \ p \to \infty, 
\end{equation}
and also by~\cite[Theorem~2]{BGS1} we know that Conjecture~\ref{conj:BGS} holds for 
all but maybe at most $X^{o(1)}$ primes $p \le X$ as $X\to \infty$. 

Here, in Theorem~\ref{thm:except set} below, 
 we obtain a more precise form of the bound~\eqref{eq:except set}.

Furthermore, Bourgain, Gamburd and Sarnak~\cite{BGS1,BGS2} have also proved that the size of any connected component
of the graphs $\cX_p$ is at least $c (\log p)^{1/3}$ 
for some absolute constant $c > 0$. 
 This bound is based on proving that any component 
contains a path of length at least $c (\log p)^{1/3}$. Here we use an additional argument and 
 show that a positive proportion of nodes along this path have ``secondary'' paths attached to 
 them which are also sufficiently long.  Finally, we show that ``many'' of the elements 
 of these ``secondary'' paths, have "tertiary"  paths that are long as well. 
 This allows us to  improve the exponent $1/3$ to $7/9$, see Theorem~\ref{thm:lower}.  

\subsection{New results}

First we improve the bound~\eqref{eq:except set}.

\begin{theorem} 
	\label{thm:except set}
	We have, 
	$$\#\(\cM_p \setminus \cC_p\) \le \exp\((\log p)^{2/3+o(1)}\) , \qquad \text{as} \ p \to \infty.
	$$
\end{theorem}

We also obtain the following improvement of a lower bound from~\cite{BGS1,BGS2}
on the size of individual components of $\cX_p$, 

\begin{theorem} 
\label{thm:lower}
The size of any connected component of $\cX_p$  is at least $c(\log p)^{7/9}$,
where $c>0$ is an absolute constant.
\end{theorem}

\section{Preliminaries}

\subsection{Solutions to  polynomial equations in subgroups 
of finite fields}

We need the following result due to Corvaja and Zannier~\cite[Corollary~2]{CoZa}, which 
generalises a series of previous estimates of a similar type, 
see~\cite{CoZa, GaVo, HBK, ShkVyu} and references therein. 

\begin{lemma}\label{lem:CZ}
Suppose that a bivariate irreducible polynomial
$$
P(X,Y) \in \ovFp[X,Y]
$$ 
is irreducible, of  degrees with respect to $X$ and $Y$
$$
\deg_X P = m \mand \deg_Y P = n,
$$
respectively,  and which is not of   the form 
$$
  \alpha X^m Y^n+\beta \qquad \text{or} \qquad \alpha X^m +\beta Y^n.
$$
with some $\alpha, \beta \in  \ovFp$. 
There exists a constant $c_0(m,n)$, depending only on $m$ and $n$, 
such that for any  multiplicative  subgroups  $\cG_1,\cG_2 \subseteq \ovFp$ 
of orders $t_1 = \# \cG_1$ and $t_2 = \# \cG_2$ we have
$$
\#\left\{ (u,v)\in \cG_1\times \cG_2 \,:\,P(u,v)=0\right\}  \le c_0(m,n) \max\{(t_1t_2)/p ,   (t_1t_2)^{1/3}\}. 
$$
\end{lemma}

\subsection{Multiplicative orders and  binary recurrences} 
For $x \in \F_p^*$ we define 
\begin{equation}
\label{eq:tx}
t(x) = \ord \xi
\end{equation}
as the order of $\xi \in \F_{p^2}^*$
which satisfies the equation  $3x=\xi+\xi^{-1}$ (it is easy to see that this is correctly defined and 
does not depend on the particular choice of $\xi$).  

Throughout the paper, as usual, we use the expressions $F \ll G$, $G \gg F $ and $F=O(G)$ to
mean that $|F|\leq \cG$ for some constant $c>0$. 

We also need the following result which follows immediately from Lemma~\ref{lem:CZ} and the explicit form of 
solutions to binary  recurrence equations.  

\begin{lem}\label{lem:small_intersect}
For  two distinct elements $x_1,x_2\in \ovFp$ we consider the binary recurrence sequences
$$
 u_{i,n+2} = 3x_iu_{i,n+1} - u_{i,n}, \qquad n = 1, 2, \ldots,
$$
 with nonzero initial values, $(u_{i,1}, u_{i,2}) \in \ovFp$, $i=1,2$.
 Then 
\begin{align*}
\# \(\{u_{1,1},\dots,u_{1,t(x_1)}\}\cap\{u_{2,1},\dots,u_{2,t(x_2)}\} \)&\\
\ll \frac{t(x_1) t(x_2)}{p} &+ \(t(x_1) t(x_2)\)^{1/3}.
\end{align*}
\end{lem}

The following statement can also be derived from~\cite[Equation~(68)]{BGS2}; however 
we give an independent proof. 

\begin{lem}\label{lem:threeorders} For any nonzero triple $(x,y,z)\in \cM_p$, 
 we have  
$$
t(x)t(y)t(z) \gg \log p.
$$ 
\end{lem}

\begin{proof}  As in~\cite{BGS1,BGS2} we note that the  inequality between the arithmetic and 
geometric means implies that the equation~\eqref{eq:Markoff}, considered over $\C$ has no 
non-zero solution $(x,y, z)$ where 
$$
3x = \xi +\xi^{-1}, \qquad 3y = \zeta+\zeta^{-1}, \qquad 3z =\eta +\eta^{-1}
$$
with the roots of unity $\xi$, $\zeta$, $\eta$ (or more generally with any $|\xi|=|\zeta|=|\eta|=1$). 

Thus if we denote by $\Phi_k$ the $k$th cyclotomic polynomial, and also define
\begin{align*}
F(U,V,W) = (U +U^{-1})^2& +(V+V^{-1})^2 +(W +W^{-1})^2 \\
& \quad -(U +U^{-1}) (V+V^{-1}) (W +W^{-1})
\end{align*}
then for any  positive integers $r,s,t$, the system of polynomials equations 
$$
U^2V^2W^2F(U,V,W) = \Phi_r(U) = \Phi_s(V) = \Phi_t(W) = 0
$$
has no solutions (unless $r=s = t =4$). 
Using the effective Hilbert's Nullstellensatz in the form given by 
D'Andrea,   Krick and  Sombra~\cite[Theorem~1]{DKS} we see that  for some 
polynomials $g_i(U,V,W) \in \Z[U,V,W]$, $i=1,\ldots, 4$ we have 
\begin{align*}
U^2V^2W^2F(U,V,W)&g_1(U,V,W)+   \Phi_r(U)g_2(U,V,W) \\
& + \Phi_s(V)g_3(U,V,W)+ \Phi_t(W)g_4(U,V,W) = A
\end{align*}
with some positive   integer $A$ with $\log A   \ll rst$. This immediately implies the result. 
\end{proof}

\subsection{Number of small divisors of integers}

For a real $z$ and an integer $n$ we use $\tau_z(n)$ to denote the number of 
integer positive divisors $d \mid n$ with $d \le z$. We present a bound on  $\tau_z(n)$ 
for small values of $z$ (which we put in a slightly more general form than we need for our applications). 

\begin{lem}\label{lem:tauz} For any fixed real positive $\gamma< 1$, if 
	$z \ge  \exp\((\log n)^{\gamma + o(1)}\)$ then 
	$$
	\tau_z(n) \le z^{1-\gamma + o(1)}
	$$
	as $n \to\infty$. 
\end{lem}

\begin{proof} As usual, we say that a positive integer is   $y$-smooth
if it is composed of prime numbers  up to $y$. Then we denote by $\psi(x,y)$ 
the number of $y$-smooth positive integers that are up to $x$.
 Let $s$ be the number of all distinct prime divisors of $n$ and let $p_1, \ldots, p_s$ 
 be the first $s$ primes. We note that 
\begin{equation}
\label{eq:tau psi}
\tau_z(n) \le \psi(z, p_s).
\end{equation}
 
By the prime number theorem we have $n \ge p_1\ldots p_s =  \exp(s + o(s))$
and thus 
\begin{equation}
\label{eq:ps z}
p_s  \ll  s \log s \le (\log n)^{1 + o(1)} \le (\log z)^{1/\gamma + o(1)}.
\end{equation}
We now recall that for any fixed $\alpha>1$ we have
$$
\Psi(x,(\log x)^\alpha)=x^{1-1/\alpha+o(1)}
$$
as $x\to \infty$, see, for example,~\cite[Equation~(1.14)]{HT}.
Combining this with~\eqref{eq:tau psi} and~\eqref{eq:ps z} we 
conclude the proof. 
\end{proof}

\section{Proofs of main results}

\subsection{Proof of Theorem~\ref{thm:except set}}

\def\X{\mathbf X} Before giving technical details we first outline the sequence 
of steps
\begin{itemize}
\item  We consider the set $\cR=\cM_p \setminus \cC_p$  and show that   
if  it is large then there is a large set $\cL \subseteq \cR$
elements of the same order $t_0$ (which is also large).

\item  Each element of $\cL$ has an orbit of size $t_0$ which is also in $\cR$.

\item We estimate the size of  intersections of these orbits for distinct elements $x_1, x_2 \in \cL$.

\item We conclude that all intersections together are small and so to fit them all 
the size of $\cR$ must be even larger than we initially assumed.

\end{itemize}

We always assume that $p$ is large enough. Define the mapping
$$
\cT_0 \(x, y, z\) \mapsto \(x, z, 3xz-y\)
$$
where $\cT_0 = \Pi_{1,3,2}\circ \cR_2$ is the composition of the permutations
$$ 
\Pi_{1,3,2} = (x, y,z)  \mapsto  (x, z,y)
$$
and the involution 
$$
\cR_2: (x, y,z) \mapsto (x, 3xz - y,  z)
$$
 as in the above.  
 
Therefore the orbit $\Gamma(x, y, z)$ of $(x, y,z)$ under the 
above group of transformations $\Gamma$ contains, in 
particular the triples $(x, u_n, u_{n+1})$, $n = 1, 2, \ldots$, 
 where the sequence $u_n$ satisfies a binary linear recurrence relation
\begin{equation}
\label{eq:bin rec}
 u_{n+2} = 3xu_{n+1} - u_n, \qquad n = 1, 2, \ldots,
\end{equation}
 with the initial values, $u_1=y$, $u_2 = z$.
 This also means that  $\Gamma(x, y, z)$ contains all triples obtained
 by the permutations of the elements in  $(x, u_n, u_{n+1})$.
 
 Let $\xi, \xi^{-1} \in \F_{p^2}^*$ be the roots of the characteristic polynomial 
 $Z^2 -3xZ +1$ of  the recurrence relation~\eqref{eq:bin rec}. 
 In particular $3x=\xi+\xi^{-1}$. Then, it is easy to see that unless  $(x, y, z)=(0,0,0)$, 
 which we eliminate from the consideration, the sequence $u_n$ is 
 periodic with period $t(x)$ which   is the order of $\xi$ in $\F_{p^2}^*$
 as given by~\eqref{eq:tx}.

We now fix some $\ve>0$ and denote
$$M_0=\exp((\log p)^{2/3+\ve}),\quad M_1 = M_0^{1/6}/2
> \exp((\log p)^{2/3+\ve/2}).$$

Assume that the remaining set of nodes $\cR=\cM_p \setminus \cC_p$ 
is of size $\# \cR>M_0$. Note that if $(x,y,z)\in \cR$ then 
also  
$\Pi(x,y,z)\in \cR$ for every $\Pi \in \rS_3$. 
 Therefore,  there are more than  $M_0^{1/3}$ elements $x\in\F_p^*$ with
$(x,y,z)\in \cR$ for some $y, z \in \F_p$. 

Since there are obviously at most $T(T+1)/2$ elements $\xi \in \F_{p^2}^*$
of order at most $T$ 
we conclude that there is a 
triple  $(x^*,y^*,z^*)\in \cR$ with 
\begin{equation}
\label{eq:large t}
t(x^*)>\sqrt{M_0^{1/3}} =2M_1.
\end{equation} 
Then the
orbit  $\Gamma(x^*,y^*,z^*)$ of this triple has at least $2M_1$ elements. Let
$M$ be the cardinality of the set $\cX$ of projections 
along the first components of all triples  $(x,y,z) \in  \Gamma(x^*,y^*,z^*)$.  
Since the orbits are closed under the permutation of coordinates, 
and permutations of the triples
$$
(x^*, u_n, u_{n+1}), \qquad n = 1,   \ldots, t(x^*),
$$ 
where the sequence $u_n$ is defined as in~\eqref{eq:bin rec} with 
respect to $(x^*,y^*,z^*)$,  produce the same projection no more than twice
we obtain 
\begin{equation}
\label{eq:M and t}
M \ge \frac{1}{2}  t(x^*).
\end{equation}
Recalling~\eqref{eq:large t}, we obtain 
\begin{equation}
\label{eq:large M}
M \ge M_1 > \exp\((\log p)^{2/3+\ve/2}\).
\end{equation}
Using that $(x,y,z)\not\in \cC_p$, we notice, that by the bound~\eqref{eq:except set}  
\begin{equation}
\label{eq:small M}
M = p^{o(1)}.
\end{equation}

By Lemma~\ref{lem:tauz}, applied with $\gamma = 2/3+\ve/2$ and the 
inequalities~\eqref{eq:large M} we have 
\begin{align*} 
\sum_{\substack{t\le M^{3/4+\ve/4}\\ t\mid p^2-1}} t 
& \le M^{3/4+\ve/4} \tau_{M^{3/4+\ve/4}}(p^2-1)\\
& =  M^{3/4+\ve/4} M^{(3/4+\ve/4)(1/3-\varepsilon/2) + o(1)} \\
& =  M^{1 - \ve/24 -\ve^2/8 + o(1)} =o(M).
\end{align*} 
For $t\mid p^2-1$
we denote $g(t)$ the number of $x\in \cX$ with $t(x)=t$. Since
$$\sum_{t\mid p^2-1}  g(t) =M$$
and $g(t)<t$ for any $t$, we conclude that
$$\sum_{\substack{t > M^{3/4+\ve/4}\\ t\mid p^2-1}}  g(t) = M + o(M)  .
$$
Next, the same argument as  used in the bound~\eqref{eq:M and t}
implies that  $g(t)=0$ for $t>2M$. Applying Lemma~\ref{lem:tauz} and the 
inequalities~\eqref{eq:large M}
again we see that for
some integer  $t_0\mid p^2-1 $ with  
\begin{equation}
 \label{eq:t and M}
2M \ge t_0 > M^{3/4+\ve/4}
\end{equation}
we have
\begin{equation}
 \label{estg}
\begin{split}
g(t_0)& \ge \frac{1}{\tau_{2M}(p^2-1)}\sum_{\substack{t > M^{3/4+\ve/3}\\ t\mid p^2-1}}  g(t) \\
& =  \frac{ M + o(M) }{\tau_{2M}(p^2-1)}  \ge   M^{2/3+\ve/2+o(1)}\ge M^{2/3+\ve/3}, 
\end{split}
\end{equation}
provided that $p$ is large enough. 

Let $\cL$ be the set of $x \in \cX$ with $t(x)=t_0$
thus 
\begin{equation}
 \label{eq: set L}
\# \cL = g(t_0).
\end{equation} 
For each $x \in \cL$ we fix some 
$y,z \in \F_p$ such $(x,y,z)\in\Gamma (x^*,y^*,z^*)$
and again consider   the sequence 
$u_n$, $n=1, 2,\ldots $,  given by~\eqref{eq:bin rec}
and of period $t(x)=t_0$, so we consider the set 
$$
\cZ(x) = \{u_n~:~n =1, \ldots, t_0\}.
$$
Let $\cH$ be the subgroup of $\F_{p^2}^*$ of order $t_0$,
and $\xi(x)$ satisfy the equation $3x=\xi(x)+\xi(x)^{-1}$. One
can easily check, using an explicit expression for binary recurrence sequences
via the roots of the characteristic polynomial, that
$$
\cZ(x) = \left\{\alpha(x) u+\frac{r (x)}{\alpha(x) u}~:~u\in \cH\right\},
$$
where 
$$r(x)=\frac{(\xi(x)^2+1)^2}{9(\xi(x)^2-1)^2},$$ and
$\alpha(x)\in\F_{p^2}^*$.
If $\xi=\xi_0$ satisfies the equation
$$r= \frac{(\xi^2+1)^2}{9(\xi^2-1)^2},$$
then other solutions are $-\xi_0, 1/\xi_0, -1/\xi_0$.
Moreover, $3x=\xi+\xi^{-1}$
can take at most two values whose sum is $0$.
Since every value is taken at most twice among the elements 
of the sequence $u_n$, $n =1, \ldots, t_0$, we have 
 \begin{equation}
 \label{eq:Z and t0}
\# \cZ(x)  \ge \frac{1}{2} t_0. 
\end{equation}

If we have   $x_1, x_2\in \cL$ with $x_1\neq \pm x_2$
(the last condition guarantees that the orbits $\cZ(x_1)$
and $\cZ(x_2)$ do not coincide),
then Lemma~\ref{lem:small_intersect} gives us the bound 
$$
\#(\cZ(x_1)\bigcap \cZ(x_2))\ll \frac{t_0^2}{p} + t_0^{2/3}. 
$$

%is the number of solutions of the equation
%$$
%\alpha(x_1) u+\frac{r (x_1)}{\alpha(x_1) u} 
%=\alpha(x_2) v+\frac{r (x_2)}{\alpha(x_2) v}
%\quad u,v\in \cH,$$ or, equivalently,
%$$
%P_{x_1,x_2}(u,v) = 0,  \qquad u,v\in \cH, 
%$$
%where 
%begin{align*} 
%P_{x_1,x_2}(X,Y) = \alpha(x_1) ^2\alpha(x_2) X^2 Y  - \alpha(&x_1)  \alpha(x_2)^2 X Y^2\\
%&  - \alpha(x_1)  r (x_2) X + \alpha(x_2)r (x_1)Y.
%\end{align*}  

We set 
 \begin{equation}
 \label{eq:h and t0}
h=\fl{c_0t_0^{1/3}}
\end{equation}
 Note that by  \eqref{estg} and~\eqref{eq: set L} we have $ \# \cL \ge M^{2/3+\ve/3} $ and recalling~\eqref{eq:t and M}
  we see that  $\# \cL  \ge h$ for an appropriate small $c_0>0$ and 
  provided that $p$ is sufficiently large.  Hence with $h$ as in~\eqref{eq:h and t0} and 
we can choose  $x_{1},\ldots,x_{h} \in \cL$ with
 $x_{i}\neq \pm x_{j}$ for $1\le i < j\le h$.    
 
We now  recall Lemma~\ref{lem:small_intersect}, which for every $j=1, \ldots, h$  we apply 
to the $j-1 < h$ intersections $\cZ(x_i)\cap \cZ(x_j)$ with $i =1, \ldots,j-1$.  

 Using that by~\eqref{eq:small M} and~\eqref{eq:t and M} we  see $t_0 = p^{o(1)}$.
Hence,  we  conclude that for $j=1,\ldots,h$ we have
 $$
\sum_{i=1}^{j-1} \#\(\cZ\(x_{j}\) \bigcap \cZ\(x_{i}\) \)  \ll h \max\{t_0^{2}/p, t_0^{2/3}\} \ll  h  t_0^{2/3}.
$$
Hence, for a sufficiently small $c_0$ in~\eqref{eq:h and t0}, recalling~\eqref{eq:Z and t0}  we derive
 $$
\sum_{i=1}^{j-1} \#\(\cZ\(x_{j}\) \bigcap \cZ\(x_{i}\) \)  \le t_0/4 \le  \frac{1}{2} \#\cZ\(x_{j}\) $$
which implies
$$
\#\(\cZ\(x_{j}\)\setminus\bigcup_{i=1}^{j-1}\cZ\(x_{i}\) \) \ge \frac{1}{2} \#\cZ\(x_{j}\) \ge t_0/4.
$$
Therefore,
$$\# \bigcup_{j=1}^{h}\cZ\(x_{j}\) \ge t_0h/4\gg t_0^{4/3}$$
and thus, by~\eqref{eq:t and M}, we have
$$\# \bigcup_{j=1}^{h}\cZ\(x_{j}\) \ >M$$
provided that $t_0$ is large enough. This contradicts the
definition of $M$.

\subsection{Proof of Theorem~\ref{thm:lower}}

We  assume that  $p$ is large enough and  fix a  connected component 
$\cC$ of $\cM_p$. 

Let $\cX$ be the set of $x\in \F_p$ such that $(x,y,z)\in\cC$ for some 
$y,z$. If $t(x)>(\log p)^{7/9}$ for some $x\in\cX$, then $\cC$ contains at least 
$t(x)$ triples $(x,y,z)$ and the desired result easily follows. Thus, we  assume
that $t(x)\le (\log p)^{7/9}$ for all $x\in \cX$. In particular, for  
$x_1, x_2 \in \cX$ 
the bound of  Lemma~\ref{lem:small_intersect}  becomes 
$O\(\(t(x_1) t(x_2)\)^{1/3}\)$. 

We consider first the case where there exists $x_0\in\cX$ such that 
\begin{equation}
\label{restrt()}
(\log p)^{0.15}\le t(x_0) \le (\log p)^{1/3}
\end{equation}
(one can see from the argument below that the exponent $0.15$ can be replaced
by any constant in the open interval $(1/7, 1/6)$).

 With every $x_0$ satisfying~\eqref{restrt()}, we associate the $t(x_0)$-periodic sequence $\{u_j\}$
as in~\eqref{eq:bin rec}. By Lemma~\ref{lem:threeorders} 
 for any $j=1, 2, \ldots$ we have 
 $$
\max\{ t(u_j),t(u_{j+1})\} \ge \sqrt{t(u_j) t(u_{j+1})} \gg  (\log p)^{1/2}t(x_0)^{-1/2}.
$$
Hence, 
if we define 
$$
\vartheta(x_0) =  c (\log p)^{1/2}t(x_0)^{-1/2} \gg  (\log p)^{1/3}
$$
for an appropriate constant $c> 0$, then for any $j=1, 2, \ldots$ we have 
$$\max\{ t(u_j),t(u_{j+1})\} \ge \vartheta(x_0).$$ 
Therefore, there are at least $t(x_0)/2$ values $j$, $1\le j\le t(x_0)$,
such that $t(u_j)\ge \vartheta(x_0)$. Since there are at most two $j$ with 
the same $t(u_j)$, there is a set $\cY(x_0)\subset \{u_1,\dots,u_{t(x_0)}\}$
with $\# \cY(x_0) \ge t(x_0)/4$ and 
\begin{equation}
\label{eq:ty yY}
t(y)\ge \vartheta(x_0), \qquad y\in \cY(x_0).
\end{equation}

We say that $y$ is associated with $x$ if $(x,y,z)\in\cC$
for some $z$. By our construction, all elements of $\cY(x_0)$ are 
associated with  $x_0$. 

Let 
$$s=\fl{c_0(\vartheta(x_0))^{1/3}}$$
where $c_0$ is a small positive constant. By the first inequality  
from~\eqref{restrt()} we have $s\le t(x_0)/4$ (provided $c_0$ 
is small enough). Hence we can choose $s$ distinct 
elements $y_1,\dots,y_s\in \cY(x_0)$. We order them so that
$$t(y_1)\le\ldots\le t(y_s).$$

For every $i=1,\ldots,s$, there is a set $\cZ(y_i)$ of elements 
associated with $y_i$ such that  
\begin{equation}
\label{eq:Zy ty}
\# \cZ(y_i)\ge t(y_i)/4
\end{equation}
 and 
\begin{equation}
\label{eq:tz}
t(z)\gg (\log p)^{1/2}t(y_i)^{-1/2}
\end{equation}
for $z\in \cZ(y_i)$.

Now we use that due to Lemma~\ref{lem:small_intersect} and the bound~\eqref{eq:ty yY}, for any $1\le j<i\le s$
we have
$$
\#\(\cZ(y_i)\cap \cZ(y_j)\) \ll (t(y_i)t(y_j))^{1/3} \ll t(y_i)^{2/3} \ll  t(y_i) \vartheta(x_0)^{-1/3}.
$$ 
Taking into account the choice of $s$  and the bound~\eqref{eq:Zy ty},  we conclude that
$$\sum_{j<i}\#\(\cZ(y_i)\cap \cZ(y_j)\) \le\frac12\cZ(y_i),
$$
provided that $c_0$ is small enough. Hence, there are   subsets 
$\cW(y_i) \subseteq \cZ(y_i)$ such that 
$$\# \cW(y_i) \ge \frac{1}{2} \# \cZ(y_i)  \ge t(y_i)/8
$$
which are pairwise disjoined, that is
$$
\cW(y_i)\cap \cW(y_j) = \emptyset, \quad 1 \le j < i \le s.
$$

For any $i=1,\dots,s$ and $z\in \cW(y_i)$ we have $t(z)$ 
triples $(x,y,z)$ from $\cC$. Summing up the bound~\eqref{eq:tz} over $z\in\cW(y_i)$ we get 
$$
\sum_{z\in\cW(y_i)} t(z) \gg (\log p)^{1/2}t(y_i)^{1/2} \ge (\log p)^{1/2}\vartheta(x_0)^{1/2}
$$
triples from $\cC$. So,
$$\# \cC \gg s (\log p)^{1/2}\vartheta(x_0)^{1/2} \gg (\log p)^{1/2}\vartheta(x_0)^{5/6}
\gg (\log p)^{7/9}$$
as required.

Now we consider the case where no element $x_0\in\cX$ satisfies~\eqref{restrt()}.
By Lemma~\ref{lem:threeorders} there exists $x_1\in\cX$ with $t(x_1)\gg(\log p)^{1/3}$.
There are at least $ t(x_1)/2$ elements $y\in\cX$ associated with $x_1$. Among them there 
are at most $(\log p)^{0.3}$ elements $y$ with $t(y)<(\log p)^{0.15}$. Hence,
there is a set $\cY(x_1)$ of elements associated with $x_1$ such that 
$\# \cY(x_1)\ge t(x_1)/3\gg(\log p)^{1/3}$ and $t(y)>(\log p)^{1/3}$ for any $y\in \cY(x_1)$. 
We now define 
$$
s=\fl{c_1(\log p)^{1/9}},
$$ 
where $c_1$ is a small positive constant, and
take elements $y_1,\dots,y_s$ from $ \cY(x_1)$. The same argument as in the
first case shows again that 
$$
\# \cC\gg  s (\log p)^{2/3} \gg (\log p)^{7/9}.
$$ 
This completes the proof.

\section*{Acknowledgement}

The authors would like to thank Peter Sarnak for the encouragement and 
useful comments as well as to the  referee  for the careful reading and valuable  suggestions.  

This work was supported by the Australian Research Council Grants DP170100786 and DP180100201 (Shparlinski) and by the Russian Science Foundation  Grant 18-41-05003 (Vyugin).

\end{document}